\documentclass [12 pt]{amsart}
\usepackage{amssymb,latexsym}
\usepackage{xypic}
\usepackage{psfrag}
\usepackage{amsfonts}
\usepackage{amssymb}
\usepackage{amsmath}
\usepackage{graphicx}
\usepackage{subfigure}
\usepackage{color}
\numberwithin{equation}{section} \numberwithin{figure}{section}

\newtheorem{definition}{Definition}

\newtheorem{theorem}{Theorem}

\newtheorem{example}{Example}
\newtheorem{proposition}{Proposition}

\numberwithin{lemma}{section} \numberwithin{theorem}{section}
\numberwithin{corollary}{section}\numberwithin{definition}{section}
\numberwithin{conjecture}{section} \numberwithin{remark}{section}
\numberwithin{example}{section} \numberwithin{table}{section}

\numberwithin{proposition}{section}

\begin{document}
\title[ Cyclic nonlinear contractions in $G$-metric spaces]{Fixed points for $G$-cyclic $\left( \phi -\psi \right)$-Kannan and $G$-cyclic $\left( \phi -\psi \right)$-Chatterjea contractions in $G$-metric spaces}
\author{Mohammad Al-Khaleel$^{*,\dagger}$ \and Sharifa Al-Sharif$^{\dagger}$}
\thanks{$^{*}$ \textbf{Corresponding author},  Department of Applied Mathematics and Sciences, Khalifa University, Abu Dhabi, UAE }
\thanks{$^{\dagger}$ Department of Mathematics, Yarmouk University, Irbid, Jordan.}
 \maketitle

\begin{abstract}
Definitions of what are called $G$-cyclic $\left( \phi
-\psi \right)$-Kannan contraction and $G$-cyclic $\left( \phi
-\psi\right)$-Chatterjea contraction are introduced in this paper. We use these new concepts to establish new fixed point results
in the context of complete generalized metric spaces. These results are new generalizations and extensions of the Kannan and Chatterjea fixed point theorems and are  generalized versions of some fixed point results proved in the literature. The analysis and theory are illustrated by some examples.
\end{abstract}

\section{\textbf{Introduction and Preliminaries}}
over the past couple of years, many researchers tried to generalize the usual metric space, see for example \cite{Gah63,Gah66,Dha92} and references therein.
Many of these generalizations were refuted by other researches \cite{Ha88,Mus04,Nai05I,Nai05II} due to the fundamental flaws they contain.
In 2006, Mustafa and Sims \cite{Mus06} were able to introduce in an appropriate new structure a new generalization called $G$-metric space in which all flaws of previous generalizations were amended. They were also able to prove that every $G$-metric space is topologically equivalent to a usual metric space, which means that it is a straightforward task to transform concepts from usual metric spaces to $G$-metric spaces. Furthermore, one can obtain similar results to those in usual metric spaces straightforwardly but in a more general setting.
The definition of the generalized metric space, $G$-metric space, as introduced by Mustafa and Sims in \cite{Mus06} is given below.
\begin{definition}[\cite{Mus06}]\label{def:G-metric}
Let $X$ be a nonempty set and let $\mathbb{R}^+$ denote the set of
all positive real numbers. Suppose that a mapping $G:X\times
X\times X\rightarrow \mathbb{R}^+$ satisfies
\begin{enumerate}
    \item [(G1)] $G(x,y,z)=0$ if $x=y=z$,
    \item [(G2)] $0< G(x,x,y)$ whenever $x\neq y$, for all $x,y\in
    X$,
    \item [(G3)] $G(x,x,y)\le G(x,y,z)$ whenever $y\neq z$, for all $x,y,z\in
    X$,
    \item [(G4)] $G(x,y,z)=G(x,z,y)=G(y,x,z)=\dots$. (Symmetry in all of the three variables),
    \item [(G5)]  $G(x,y,z)\le G(x,a,a)+G(a,y,z)$, for all $x,y,z,a\in
    X$. (Rectangle inequality).
\end{enumerate}
Then $G$ is called a generalized metric, $G$-metric on $X$, and
$(X,G)$ is called a generalized metric space, $G$-metric space.
\end{definition}
Below, we give some examples of the $G$-metric spaces as well as some other definitions
and convergence properties of sequences in $G$-metric spaces.
\begin{example}[\cite{Mus06}] Let $(X,d)$ be any metric space. Define $G_s$ and $G_m$ on $X\times X\times X$ to $\mathbb{R}^+$ by
$$
\begin{array}{l}
G_s(x,y,z)=d(x,y)+d(y,z)+d(x,z),\\
G_m(x,y,z)=\max\{d(x,y),d(y,z),d(x,z)\}, \forall x,y,z\in X.
\end{array}
$$
Then, $(X,G_s)$ and $(X,G_m)$ are $G$-metric spaces.
\end{example}
\begin{definition}[ \cite{Mus06}]
Let $(X,G)$ be a $G$-metric space and $\{x_n\}$ be a sequence of
points in $X$. Then
\begin{enumerate}
\item [(i)] a point $x\in X$ is said to be the limit of the
sequence $\{x_n\}$ if for any $\varepsilon>0$, there exists an
integer $n\in\mathbb{N}$ such that $G(x,x_n,x_m)<\epsilon$, for
all $n,m\ge N$, i.e., if
$$ \lim\limits_{n,m\rightarrow\infty} G(x,x_n,x_m)=0,
$$
and we say that the sequence $\{x_n\}$ is $G$-convergent to $x$.
\item [(ii)] the sequence $\{x_n\}$ is said to be $G$-Cauchy if
any given $\varepsilon>0$, there is $n\in\mathbb{N}$ such that
$G(x_n,x_m,x_\ell)<\epsilon$, for all $n,m,\ell\ge N$, i.e., if
$$
\lim\limits_{n,m,\ell\rightarrow\infty}G(x_n,x_m,x_{\ell})=0.
$$
\item [(iii)] the space $(X,G)$ is called a complete $G$-metric
space if every $G$-Cauchy sequence in $X$ is $G$-convergent in
$X$.
\end{enumerate}
\end{definition}
\begin{proposition}[ \cite{Mus06}]
Let $(X,G)$ be a $G$-metric space. Then the following are
equivalent.
\begin{enumerate}
\item The sequence $\{x_n\}$ is $G$-convergent to $x$.
\item $\lim\limits_{n\rightarrow\infty}G(x_n,x_n,x)=0$.
\item $\lim\limits_{n\rightarrow\infty}G(x_n,x,x)=0$.
\item $\lim\limits_{n,m\rightarrow\infty}G(x_n,x_m,x)=0$.
\end{enumerate}
\end{proposition}
\begin{proposition}[ \cite{Mus06}]
Let $(X,G)$ be a $G$-metric space. Then the following are
equivalent.
\begin{enumerate}
\item The sequence $\{x_n\}$ is $G$-Cauchy in $X$.
\item $\lim\limits_{n,m\rightarrow\infty}G(x_n,x_m,x_m)=0$.
\end{enumerate}
\end{proposition}
Several fixed point theorems and results were obtained in this new generalization of the usual metric spaces,
see for example \cite{Alkhaleel12}-\cite{Alkhaleel17} and references therein.

Meanwhile, there have been also several attempts to extend and generalize the Banach contraction mapping
principle \cite{Banach22} in the usual metric spaces over the past couple of years.
Kannan \cite{Kan68} successfully extended the well-known Banach's contraction principle \cite{Banach22} by proving that if $X$ is complete, then every what is so-called Kannan contraction $T$ has a unique fixed point.
The definition that was introduced by Kannan is stated below.
\begin{definition}[\cite{Kan68}]
A mapping $T:X\rightarrow X$, where $\left( X,d\right) $ is a
metric space, is said to be a Kannan contraction if there exists
$\alpha \in \left[ 0,\frac{1}{2}\right) $ such that for all
$x,y\in X$, the inequality
\begin{equation*}
d\left( Tx,Ty\right) \leq \alpha \left[ d\left( x,Tx\right)
+d\left( y,Ty\right) \right],
\end{equation*}
holds.
\end{definition}
Another definition that was also used to extend the well-known Banach's contraction principle \cite{Banach22} which is a sort of dual of Kannan contraction, is
presented by Chatterjea \cite{Chat72} as follows.
\begin{definition}[\cite{Chat72}]
A mapping $T:X\rightarrow X$, where $\left( X,d\right)$ is a
metric space, is said to be a Chatterjea contraction if there
exists $\alpha \in \left[ 0,\frac{1}{2}\right) $ such that for all
$x,y\in X$, the inequality
\begin{equation*}
d\left( Tx,Ty\right) \leq \alpha \left[ d\left( x,Ty\right)
+d\left( y,Tx\right) \right],
\end{equation*}
holds.
\end{definition}
Chatterjea \cite{Chat72} also proved using his new definition that if $X$ is complete, then
every Chatterjea contraction has a unique fixed point.
In 1972, Zamfirescu \cite{Zam72} introduced a very interesting fixed point theorem which
combines the contractive conditions of Banach, Kannan, and Chatterjea.
\begin{theorem}[\cite{Zam72}]
Let $(X,d)$ be a complete metric space and $T:X\rightarrow X$ a map for which there exist the real numbers $\alpha ,\beta$, and
$\gamma $ satisfying $0\leq \alpha <1$, $0\leq \beta ,\gamma<
\frac{1}{2}$, such that for $x,y\in X$ at least one of the
following is true.
\begin{enumerate}
\item[(i)] $d\left( Tx,Ty\right) \leq \alpha d\left( x,y\right)$,
\item[(ii)] $d(Tx,Ty)\leq \beta \left[ d\left( x,Tx\right)
+d\left( y,Ty\right) \right]$,
\item[(iii)] $d\left( Tx,Ty\right) \leq \gamma \left[ d\left(
x,Ty\right) +d\left( y,Tx\right) \right]$.
\end{enumerate}
Then $T$ has a unique fixed point $p$ and the Picard iteration $\{x_n\}_{n=0}^{\infty}$ defined by $x_{n+1}=Tx_n, \;n=0,1,2,\dots$ converges to $p$ for any $x_0\in X$.
\end{theorem}

The cyclical extensions for these fixed point theorems were
obtained at a later time, by considering non-empty closed subsets
$\left\{ A_{i}\right\} _{i=1}^{p}$ of a complete metric space $X$
and a cyclical operator
$T:\bigcup\limits_{i=1}^{p}A_{i}\rightarrow
\bigcup\limits_{i=1}^{p}A_{i}$, i.e., satisfies $T\left(
A_{i}\right) \subseteq A_{i+1}\text{ for all }i\in \left\{
1,2,\dots,p\right\}$.
In \cite{Rus05}, Rus presented the cyclical extension for the
Kannan's theorem, and Petric in \cite{Pet10} presented cyclical
extensions for Chatterjea and Zamfirescu theorems using fixed
point structure arguments.
The concept of a control function in terms of altering distances was addressed by Khan
\emph{et. al.} \cite{Khan84} which lead to a new category of fixed point problems.
Altering distances have been used in metric fixed point theory in
many papers, see for example \cite{Sas99}-\cite{Naidu03} and
references therein.
In this paper, we consider the generalization of the usual metric space introduced in \cite{Mus06}, $G$-metric space, and study new extensions and generalizations of Banach, Kannan, and Chatterjea contractions to present and prove new fixed point theorems. We give some generalized versions of the fixed point results proved in the literature in the context of $G$-metric spaces. In particular, we present some generalized versions of fixed point
theorems of cyclic nonlinear contractions type in $G$-metric spaces by the use of the continuous function $\psi$
and the altering distance function $\phi$  which are both defined below.
At the end of this paper we illustrate the analysis and the theory by some examples.

\begin{definition}\label{def:altering distance}
The function $\phi :\left[ 0,\infty \right) \rightarrow \left[
0,\infty \right) $ is called an altering distance function, if the
following properties are satisfied.
\begin{enumerate}
\item[(i)] $\phi $ is continuous,
\item[(ii)]$\phi$ is nondecreasing,
\item[(iii)] $\phi \left( t\right) =0$ if and only if $t=0$.
\end{enumerate}
The function $\psi :\left[ 0,\infty \right) ^{3}\rightarrow \left[
0,\infty \right) $ is a continuous function such that $\psi \left(
x,y,z\right) =0$ if and only if $x=y=z=0$.
\end{definition}

\section{\textbf{ Main results}}
We begin this section by giving definitions of what we call a
$G$-cyclic $\left( \phi -\psi \right) $-Kannan type contraction
and a $G$-cyclic $\left(\phi-\psi\right)$-Chatterjea type
contraction.
\begin{definition}\label{def:cycl-khan-Ch}
Let $\left\{ A_{i}\right\} _{i=1}^{p}$ be non-empty closed subsets
of a $G$-metric space $(X,G)$, and suppose $T:\bigcup
\limits_{i=1}^{p}A_{i}\rightarrow \bigcup\limits_{i=1}^{p}A_{i}$
is a cyclical operator.

Then $T$ is said to be a $G$-cyclic $\left( \phi -\psi \right)
$-Kannan type contraction if there exists constants $\alpha,\gamma$ with
$0\le\gamma<1$ and $0<\alpha+\gamma\le1$, such that for any $x\in
A_{i},y,z\in A_{i+1},i=1,2,\dots,p$, we have
\begin{eqnarray*}
\phi \left( G\left( Tx,Ty,Tz\right) \right) \leq \phi \left(
\alpha G\left( x,Tx,Tx\right) +\beta( G\left( y,Ty,Ty\right)+
G\left( z,Tz,Tz\right))\right)\\
-\psi \left( G\left( x,Tx,Tx\right),G\left( y,Ty,Ty\right),G\left(
z,Tz,Tz\right) \right),
\end{eqnarray*}
and $T$ is said to be a $G$-cyclic
$\left(\phi-\psi\right)$-Chatterjea type contraction if there
exists constants $\alpha,\beta$ with $0\le\alpha\le\frac{1}{2}$
and $0<\alpha+\beta\le 1$, such that for any $x\in A_{i},y,z\in
A_{i+1},i=1,2,\dots,p$, we have
$$
\begin{array}{ll}
\phi \left( G\left( Tx,Ty,Tz\right) \right) \leq \phi \left(
\alpha G\left( x,Ty,Tz\right) +\beta G\left( y,z,Tx\right)
\right)\\
\hspace{5cm} -\psi \left( G\left( x,Ty,Tz\right) ,G\left(
y,z,Tx\right),G\left( z,y,Tx\right) \right),
\end{array}
$$
where $\phi :\left[ 0,\infty \right) \rightarrow \left[ 0,\infty
\right) $ and $\psi :\left[ 0,\infty \right) ^{3}\rightarrow
\left[ 0,\infty \right) $ are given in Definition
\ref{def:altering distance}.
\end{definition}

\begin{theorem}\label{th:main1}
Let $\left\{ A_{i}\right\} _{i=1}^{p}$ be non-empty closed subsets
of a complete $G$-metric space $\left( X,G\right)$ and
$T:\bigcup\limits_{i=1}^{p}A_{i}\rightarrow\bigcup\limits_{i=1}^{p}A_{i}$
satisfies at least one of the following:
\begin{enumerate}

\item There exists constants $\alpha,\gamma$ with
$0\le\gamma<1$ and $0<\alpha+\gamma\le1$, such that for any $x\in A_{i},y\in
A_{i+1},i=1,2,\dots,p$, we have
$$\begin{array}{c} \phi \left(
G\left( Tx,Ty,Ty\right) \right) \leq \phi \left(
\alpha G\left( x,Tx,Tx\right) +\gamma G\left( y,Ty,Ty\right)\right)\\
\hspace{6cm}-\psi \left( G\left( x,Tx,Tx\right),G\left(
y,Ty,Ty\right),G\left( y,Ty,Ty\right) \right).
\end{array}$$
\item There exists constants $\alpha,\delta$ with
$0\le\alpha\le\frac{1}{2}$ and $0<\alpha+\delta\le 1$, such that
for any $x\in A_{i},y\in A_{i+1},i=1,2,\dots,p$, we have
$$
    \begin{array}{c}
\phi \left( G\left( Tx,Ty,Ty\right) \right) \leq \phi \left(
\alpha G\left( x,Ty,Ty\right) +\delta G\left( y,y,Tx\right)
\right)
\\
\hspace{6cm}-\psi \left( G\left( x,Ty,Ty\right) ,G\left(
y,y,Tx\right),G\left( y,y,Tx\right) \right).
\end{array}
$$
\end{enumerate}
Then $T$ has a unique fixed point $u\in
\bigcap\limits_{i=1}^{p}A_{i}$.
\end{theorem}
\begin{proof}
Take $x_{0}\in X$ and consider the sequence given by
$x_{n+1}=Tx_{n},n\geq 0$.
If there exists $n_{0}\in N$ such that $x_{n_{0}+1}=x_{n_{0}}$,
then the existence of the fixed point is proved.
So, suppose that $x_{n+1}\neq x_{n}$ for any $n=0,1,\dots$. Then
there exists $i_{n}\in \left\{ 1,\dots,p\right\}$ such that
$x_{n-1}\in A_{i_{n}}$ and $x_{n}\in A_{i_{n+1}}$.

Now, assume first that $T$ satisfies condition $(1)$. Then, we
have
\begin{eqnarray*}
\phi \left( G\left( x_{n},x_{n+1},x_{n+1}\right) \right) &=&\phi
\left(
G\left(Tx_{n-1},Tx_{n},Tx_{n}\right) \right)\\
&\leq& \phi \left( \alpha G\left( x_{n-1},Tx_{n-1},Tx_{n-1}\right)
+\gamma G\left( x_{n},Tx_{n},Tx_{n}\right)
\right)\\
&\;\;-&\psi \left( G\left(x_{n-1},Tx_{n-1},Tx_{n-1}\right),G\left( x_{n},Tx_{n},Tx_{n}\right), G\left( x_{n},Tx_{n},Tx_{n}\right)\right) \\
&= &\phi \left( \alpha G\left(x_{n-1},x_{n},x_{n}\right) +\gamma
G\left( x_{n},x_{n+1},x_{n+1}\right)
\right)\\
&\;\;-&\psi \left( G\left(x_{n-1},x_{n},x_{n}\right),G\left( x_{n},x_{n+1},x_{n+1}\right),G\left( x_{n},x_{n+1},x_{n+1}\right) \right)  \\
&\leq& \phi \left( \alpha G\left( x_{n-1},x_{n},x_{n}\right)
+\gamma G\left( x_{n},x_{n+1},x_{n+1}\right) \right).
\end{eqnarray*}
Since $\phi $ is a nondecreasing function, we get
$$
G\left( x_{n},x_{n+1},x_{n+1}\right) \leq \alpha G\left(
x_{n-1},x_{n},x_{n}\right) +\gamma G\left(
x_{n},x_{n+1},x_{n+1}\right),
$$
which implies
\begin{eqnarray}\label{eq:1}
 G\left( x_{n},x_{n+1},x_{n+1}\right)&\leq &\frac{\alpha}{1-\gamma}G\left( x_{n-1},x_{n},x_{n}\right) ,\forall n.
\end{eqnarray}
Since $0<\alpha+\gamma\le1$, we get that $ G\left(
x_{n},x_{n+1},x_{n+1}\right)$ is a nonincreasing sequence of
nonnegative real numbers.
Hence, there is $r\geq 0$ such that
\begin{equation*}
\lim_{n\rightarrow \infty }G\left( x_{n},x_{n+1},x_{n+1}\right)
=r.
\end{equation*}
Using the continuity of $\phi $ and $\psi $, we get
\begin{eqnarray*}
\phi \left( r\right) &\leq &\phi \left( (\alpha+\gamma)r\right)
-\psi \left(
r,r,r\right) \\
&\leq &\phi \left( r\right) -\psi \left( r,r,r\right),
\end{eqnarray*}
which implies that $\psi \left( r,r,r\right) =0$, and hence,
$r=0$.
\\

Similarly, if $T$ satisfies condition $(2)$, then we have
\begin{eqnarray*}
\phi \left( G\left( x_{n},x_{n+1},x_{n+1}\right) \right) &=&\phi
\left(
G\left(Tx_{n-1},Tx_{n},Tx_{n}\right) \right) \\
&\leq& \phi \left( \alpha G\left( x_{n-1},Tx_{n},Tx_{n}\right)
+\gamma G\left( x_{n}, x_{n},Tx_{n-1}\right) \right)\\
&& -\psi \left(
G\left(x_{n-1},Tx_{n},Tx_{n}\right),G\left(x_{n},x_{n},Tx_{n-1}\right),G\left(x_{n},x_{n},Tx_{n-1}\right) \right) \\
&=& \phi \left( \alpha G\left( x_{n-1},x_{n+1},x_{n+1}\right)
+\gamma G\left(x_n,x_{n},x_{n}\right)\right)\\
&&-\psi \left( G\left(
x_{n-1},x_{n+1},x_{n+1}\right),G\left(x_n,x_{n},x_{n}\right),G\left(x_n,x_{n},x_{n}\right)\right)  \\
&\leq& \phi \left( \alpha G\left( x_{n-1},x_{n+1},x_{n+1}\right)
\right).
\end{eqnarray*}
Since, $\phi $ is a nondecreasing function, we get
\begin{equation}\label{eq:2}
G\left( x_{n},x_{n+1},x_{n+1}\right) \leq\alpha
G\left(x_{n-1},x_{n+1},x_{n+1}\right),
\end{equation}
and by rectangular inequality, we have
\begin{eqnarray*}
G\left( x_{n},x_{n+1},x_{n+1}\right) &\leq& \alpha G\left(x_{n-1},x_{n+1},x_{n+1}\right)\\
&\leq &\alpha \left[ G\left( x_{n-1},x_{n},x_n\right)+G\left(
x_{n},x_{n+1},x_{n+1}\right) \right],
\end{eqnarray*}
which implies
\begin{eqnarray}\label{eq:3}
G\left( x_{n},x_{n+1},x_{n+1}\right) &\leq&
\frac{\alpha}{1-\alpha}G\left( x_{n-1},x_{n},x_n\right).
\end{eqnarray}
Since $0\le\alpha\leq\frac{1}{2} $, we get that $\left\{ G\left(
x_{n},x_{n+1},x_{n+1}\right) \right\} $ is a nonincreasing
sequence of nonnegative real numbers.
Hence, there is $r\geq 0$ such that
\begin{equation*}
\lim_{n\rightarrow \infty }G\left( x_{n},x_{n+1},x_{n+1}\right)
=r.
\end{equation*}
Now, if $\alpha=0$, then clearly, $r=0$, and if $0<\alpha<
\frac{1}{2}$, then $\frac{\alpha}{1-\alpha}<1$, and by induction,
we have
$$
G\left( x_{n},x_{n+1},x_{n+1}\right) \leq
\left(\frac{\alpha}{1-\alpha}\right)^nG\left(
x_{0},x_{1},x_1\right),
$$
and hence, $r=0$.
Finally, if $\alpha=\frac{1}{2}$, then from (\ref{eq:2}), we have
$$
G\left( x_{n-1},x_{n+1},x_{n+1}\right) \ge 2G\left(
x_{n},x_{n+1},x_{n+1}\right),
$$
and hence,
$$
\lim\limits_{n\rightarrow\infty}G\left(
x_{n-1},x_{n+1},x_{n+1}\right)\ge 2r,
$$
but,
\begin{equation*}
G\left( x_{n-1},x_{n+1},x_{n+1}\right) \leq G\left(
x_{n-1},x_{n},x_n\right)+G\left( x_{n},x_{n+1},x_{n+1}\right),
\end{equation*}
and as $n\rightarrow\infty$, we have
$$
\lim\limits_{n\rightarrow\infty}G\left(
x_{n-1},x_{n+1},x_{n+1}\right)\leq2r.
$$
Therefore, $\lim\limits_{n\rightarrow \infty }G\left(
x_{n-1},x_{n+1},x_{n+1}\right)=2r$.
Using the continuity of $\phi $ and $\psi $, and
$\alpha=\frac{1}{2}$, we get
\begin{eqnarray*}
\phi \left( r\right) &\leq &\phi \left( \frac{1}{2}\cdot 2r\right)
-\psi \left(
2r,0,0\right)\\
&= &\phi \left( r\right) -\psi \left( 2r,0,0\right),
\end{eqnarray*}
which implies that $\psi \left( 2r,0,0\right) =0$, and hence,
$r=0$.

In the sequel, we show that $\left\{ x_{n}\right\} $ is a
$G$-Cauchy sequence in $X$.
To do so, we need to prove first, the claim that for every
$\epsilon>0$, there exists $n\in\mathbb{N}$ such that if $p,q\geq
n$ with $p-q\equiv 1\left( m\right)$, then $G\left( x_{p},x_{q},
x_q\right) <\epsilon $.
Suppose the contrary case, i.e., there exists $\epsilon
>0 $ such that for any $n\in\mathbb{N}$, we can find $p_{n}>q_{n}\geq n$ with $p_{n}-q_{n}\equiv 1\left( m\right)$
satisfying $G\left( x_{p_{n}},x_{q_{n}},x_{q_{n}}\right) \geq
\epsilon $.
Now, we take $n>2m$. Then corresponding to $q_{n}\geq n$, we can
choose $p_{n}$ in such a way that it is the smallest integer with
$p_{n}>q_{n}$ satisfying $ p_{n}-q_{n}\equiv 1\left( m\right) $
and $G\left( x_{p_{n}},x_{q_{n}},x_{q_{n}}\right)\geq \epsilon $.
Therefore, $G\left( x_{q_{n}},x_{q_n},x_{p_{n-m}}\right)
<\epsilon$.
Using the rectangular inequality,
\begin{eqnarray*}
\epsilon \leq G\left( x_{p_{n}},x_{q_{n}},x_{q_{n}}\right) &\leq&
G\left( x_{q_{n}},x_{q_{n}},x_{p_{n-m}}\right)
+\sum\limits_{i=1}^{m}G\left(x_{p_{n-i}},
x_{p_{n-i}},x_{p_{n-i+1}}\right) \\
&<&\epsilon +\sum\limits_{i=1}^{m}G\left( x_{p_{n-i}},
x_{p_{n-i}},x_{p_{n-i+1}}\right).
\end{eqnarray*}
Letting $n\rightarrow \infty $ in the last inequality, and taking
into account that \\
$\lim\limits_{n\rightarrow\infty}G\left(x_{n},x_{n+1},x_{n+1}\right)
=0$, we obtain $\lim\limits_{n\rightarrow \infty }G\left(
x_{p_{n}},x_{q_{n}},x_{q_{n}}\right) =\epsilon $.
Again, by rectangle inequality, we have
\begin{eqnarray*}
G(x_{q_{n}},x_{q_{n}},x_{p_{n}}) &\le&
G(x_{p_{n}},x_{p_{n+1}},x_{p_{n+1}})+G(x_{p_{n+1}},x_{q_{n}},x_{q_{n}})\\
&\le&
G(x_{p_{n}},x_{p_{n+1}},x_{p_{n+1}})+G(x_{p_{n+1}},x_{q_{n+1}},x_{q_{n+1}})+G(x_{q_{n+1}},x_{q_{n}},x_{q_{n}})\\
&\le&
G(x_{p_{n}},x_{p_{n+1}},x_{p_{n+1}})+G(x_{p_{n+1}},x_{q_{n+1}},x_{q_{n+1}})+G(x_{q_{n}},x_{q_{n+1}},x_{q_{n+1}})\\
&&+G(x_{q_{n+1}},x_{q_{n+1}},x_{q_{n}}),
\end{eqnarray*}
moreover,
\begin{eqnarray*}
G(x_{p_{n+1}},x_{q_{n+1}},x_{q_{n+1}})
&\le&G(x_{p_{n+1}},x_{q_{n}},x_{q_{n}})+G(x_{q_{n}},x_{q_{n+1}},x_{q_{n+1}})\\
&\le&
G(x_{p_{n+1}},x_{p_{n}},x_{p_{n}})+G(x_{p_{n}},x_{q_{n}},x_{q_{n}})+G(x_{q_{n}},x_{q_{n+1}},x_{q_{n+1}})\\
&\le&G(x_{p_{n}},x_{p_{n+1}},x_{p_{n+1}})+G(x_{p_{n+1}},x_{p_{n+1}},x_{p_{n}})+G(x_{p_{n}},x_{q_{n}},x_{q_{n}})\\
&&+G(x_{q_{n}},x_{q_{n+1}},x_{q_{n+1}}).
\end{eqnarray*}
Taking the limit as $n\rightarrow\infty$, and taking into account
that
$\lim\limits_{n\rightarrow\infty}G\left(x_{n},x_{n+1},x_{n+1}\right)
=0$, we get $\epsilon\le\lim\limits_{n\rightarrow \infty }G\left(
x_{p_{n+1}},x_{q_{n+1}},x_{q_{n+1}}\right) \le\epsilon $, which
implies that $\lim\limits_{n\rightarrow \infty }G\left(
x_{p_{n+1}},x_{q_{n+1}},x_{q_{n+1}}\right) =\epsilon $.

Since $x_{p_{n}}$ and $x_{q_{n}}$ lie in different adjacently
labelled sets $A_{i}$ and $A_{i+1}$ for certain $1\leq i\leq m$,
assuming that $T$ satisfies condition $(1)$, we have
\begin{eqnarray*}
\phi \left( G\left( x_{q_{n+1}},x_{q_{n+1}},x_{p_{n+1}}\right)
\right) &=&\phi
\left(G\left( Tx_{q_{n}},Tx_{q_{n}},Tx_{p_{n}}\right) \right)  \\
&\leq&\phi \left( \alpha G\left(
x_{q_{n}},Tx_{q_{n}},Tx_{q_{n}}\right)
+\gamma G\left( x_{p_{n}},Tx_{p_{n}},Tx_{p_{n}}\right) \right)\\
&-&\psi \left(G\left(
x_{p_{n}},Tx_{p_{n}},Tx_{p_{n}}\right),G\left(
x_{q_{n}},Tx_{q_{n}},Tx_{q_{n}}\right),G\left(
x_{q_{n}},Tx_{q_{n}},Tx_{q_{n}}\right)\right).
\end{eqnarray*}
Letting $n\rightarrow \infty $ in the last inequality, we obtain
\begin{equation*}
\phi \left( \epsilon \right) \leq \phi \left(0 \right) -\psi
\left( 0,0,0 \right)=0.
\end{equation*}
Therefore, we get $\epsilon =0$ which is a contradiction.

Similarly, assuming that $T$ satisfies condition $(2)$, we have
\begin{eqnarray*}
\phi \left( G\left( x_{q_{n+1}}, x_{q_{n+1}},x_{p_{n+1}}\right)
\right) &=&\phi \left(
G\left( Tx_{q_{n}}, Tx_{q_{n}},Tx_{p_{n}}\right) \right) \\
&\leq &\phi \left( \alpha G\left(
x_{p_{n}},Tx_{q_{n}},Tx_{q_{n}}\right) +\gamma G\left(
x_{q_{n}},x_{q_{n}},Tx_{p_{n}}\right)
\right)\\
&&-\psi \left( G\left(x_{p_{n}},
Tx_{q_{n}},Tx_{q_{n}}\right),G\left(
x_{q_{n}},x_{q_{n}},Tx_{p_{n}}\right),G\left(
x_{q_{n}},x_{q_{n}},Tx_{p_{n}}\right) \right).
\end{eqnarray*}
Letting $n\rightarrow \infty $ in the last inequality, we obtain
\begin{equation*}
\phi \left( \epsilon \right) \leq \phi \left(
(\alpha+\gamma)\epsilon \right) -\psi \left( \epsilon
,\epsilon,\epsilon \right).
\end{equation*}
Therefore, since $0<\alpha+\gamma\le1$, we get $\psi \left(
\epsilon ,\epsilon,\epsilon \right) =0$, and hence, $ \epsilon
=0$, which is a contradiction.

From the above proved claim for both cases, i.e., the case when
$T$ satisfies condition $(1)$ and the case when $T$ satisfies
condition $(2)$, and for arbitrary $\epsilon>0$, we can find
$n_0\in \mathbb{N}$ such that if $p,q>n_0$ with $p-q=1(m)$, then
$G\left( x_{p},x_{q},x_{q}\right)<\epsilon$.

Since $\lim\limits_{n\rightarrow\infty}G(x_n,x_{n+1},x_{n+1})=0$,
we can find $n_1\in\mathbb{N}$ such that
$$
G(x_n,x_{n+1},x_{n+1})\le \frac{\epsilon}{m},\; \mbox{for}\;
n>n_1.
$$
Now, for $r,s>\max\{n_0,n_1\}$ and $s>r$, there exists $k\in
\{1,2,\dots,m\}$ such that $s-r=k(m)$.
Therefore, $s-r+j=1(m)$ for $j=m-k+1$.
So, we have
$$
G(x_r,x_r,x_s)\le
G(x_r,x_r,x_{s+j})+G(x_{s+j},x_{s+j},x_{s+j-1})+\dots+G(x_{s+1},x_{s+1},x_{s}).
$$
This implies
$$
G(x_r,x_r,x_s)\le \epsilon+\frac{\epsilon}{m}\sum\limits_{j=1}^m
\;1=2\epsilon.
$$
Thus, $\left\{ x_{n}\right\} $ is a $G$-Cauchy sequence in
$\bigcup\limits_{i=1}^{p}A_{i}$.
Consequently, $\left\{ x_{n}\right\} $ converges to some $u\in
\bigcup\limits_{i=1}^{p}A_{i}$.
However, in view of cyclical condition, the sequence
$\left\{x_{n}\right\} $ has an infinite number of terms in each
$A_{i}$, for $i=1,2,\dots,p$.
Therefore, $u\in \bigcap\limits_{i=1}^{p}A_{i}$.

Now, we will prove that $u$ is a fixed point of $T$.
Suppose $u\in A_{i}$, $Tu\in A_{i+1}$, and we take a subsequence
$x_{n_{k}}$ of $\left\{ x_{n}\right\}$ with $x_{n_{k}}\in
A_{i-1}$.
Then, assuming that $T$ satisfies condition $(1)$, we have
\begin{eqnarray*}
\phi \left( G\left( x_{n_{k+1}},Tu,Tu\right) \right) &=&\phi
\left(
G\left(Tx_{n_{k}},Tu,Tu\right) \right)  \\
&\leq& \phi\left(\alpha
G\left(x_{n_{k}},Tx_{n_{k}},Tx_{n_{k}}\right)
+\gamma G\left( u,Tu,Tu\right) \right)\\
&& -\psi \left( G\left(x_{n_{k}},Tx_{n_{k}},Tx_{n_{k}}\right) ,G\left( u,Tu,Tu\right),G\left( u,Tu,Tu\right) \right) \\
&\leq&\phi\left(\alpha
G\left(x_{n_{k}},Tx_{n_{k}},Tx_{n_{k}}\right) +\gamma G\left(
u,Tu,Tu\right)  \right).
\end{eqnarray*}
Letting $k\rightarrow \infty $, we have
\begin{equation*}
\phi \left( G\left( u,Tu,Tu\right) \right) \leq \phi \left( \alpha
G\left( u,u,u\right) +\gamma G\left( u,Tu,Tu\right) \right),
\end{equation*}
and since $\phi $ is a nondecreasing function, we get
\begin{equation*}
G\left( u,Tu,Tu\right) \leq \gamma G\left( u,Tu,Tu\right).
\end{equation*}
Thus, since $0\le\gamma< 1$, we have $G\left(
u,Tu,Tu\right) =0$, and hence, $u=Tu$.
\\

Similarly, assuming that $T$ satisfies condition $(2)$, then we
have
\begin{eqnarray*}
\phi \left( G\left( x_{n_{k+1}},Tu,Tu\right) \right) &=&\phi
\left( G\left(
Tx_{n_{k}},Tu,Tu\right) \right) \\
&\leq &\phi \left( \alpha G\left( x_{n_{k}},Tu,Tu\right) +\gamma
G\left( u,u,Tx_{n_{k}}\right) \right)\\
 && -\psi \left( G\left(
x_{n_{k}},Tu,Tu\right),G\left( u,u,Tx_{n_{k}}\right),G\left( u,u,Tx_{n_{k}}\right) \right)  \\
&\leq &\phi \left( \alpha G\left( x_{n_{k}},Tu,Tu\right) +\gamma
G\left( u,u,Tx_{n_{k}}\right) \right).
\end{eqnarray*}
Letting $k\rightarrow \infty $, we have
\begin{equation*}
\phi \left( G\left( u,Tu,Tu\right) \right) \leq \phi \left( \alpha
G\left( u,Tu,Tu\right) +\gamma G\left( u,u,u\right) \right),
\end{equation*}
since $\phi $ is a nondecreasing function, we get
\begin{equation*}
G\left( u,Tu,Tu\right) \leq \alpha G\left( u,Tu,Tu\right).
\end{equation*}
Thus, since $0\le\alpha\le \frac{1}{2}$, we have $G\left(
u,Tu,Tu\right) =0$, and hence, $u=Tu$.
\end{proof}

\begin{theorem}\label{th:main}
Let $\left\{ A_{i}\right\} _{i=1}^{p}$ be non-empty closed subsets
of a complete $G$-metric space $\left( X,G\right)$ and
$T:\bigcup\limits_{i=1}^{p}A_{i}\rightarrow\bigcup\limits_{i=1}^{p}A_{i}$
be at least one of the following.
\begin{enumerate}
    \item a $G$-cyclic $\left( \phi -\psi \right)$-Kannan type
    contraction.
    \item a $G$-cyclic $\left( \phi -\psi \right)$-Chatterjea type
    contraction.
\end{enumerate}
Then $T$ has a unique fixed point $u\in
\bigcap\limits_{i=1}^{p}A_{i}$.
\end{theorem}
\begin{proof}
Taking $z=y$ in Definition \ref{def:cycl-khan-Ch}, the proof
follows straightforwardly from the proof of Theorem \ref{th:main1}
with $\gamma=2\beta$ for the first condition and $\delta=\beta$
for the second condition.
\end{proof}

\section{Applications and Examples}
We give below two examples in order to validate the proved result.
\\
\begin{example} Let $X$ be a complete $G$-metric space, $m$
positive integer, $A_1,\dots,A_m$ non-empty closed subsets of $X$,
and $X=\bigcup\limits_{i=1}^m A_i$.
Let $T:X\rightarrow X$ be an operator such that
\begin{enumerate}
    \item [(i)] $X=\bigcup\limits_{i=1}^m A_i$ is a cyclic
    representation of $X$ with respect to $T$.
     \item [(ii)] for any $x\in A_i$, $y\in A_{i+1}$,
    $i=1,2,\dots,m$, where $A_{m+1}=A_1$ and $\rho:[0,\infty)\rightarrow
    [0,\infty)$ is a Lebesgue integrable mapping satisfies
    $\int^t_0 \rho(s) \;ds>0$ for $t>0$, we have one of the
    following:
    $$
    \int_0^{G(Tx,Ty,Ty)}\rho(t)\;dt \le \int_0^{\alpha
    G(x,Tx,Tx)+\gamma  G(y,Ty,Ty)}\rho(t)\;dt,
    $$
    or
    $$
    \int_0^{G(Tx,Ty,Ty)}\rho(t)\;dt \le \int_0^{\alpha
    G(x,Ty,Ty)+\gamma  G(Tx,y,y)}\rho(t)\;dt.
    $$
Then $T$ has a unique fixed point $u\in \bigcap\limits_{i=1}^m
A_i$.
\end{enumerate}
In order to see this, one might let $\phi:[0,\infty)\rightarrow
    [0,\infty)$ be defined as $\phi(t)=\int^t_0 \rho(s) \;ds>0$.
    Then, $\phi$ is alternating distance function, and by taking
    $\psi(t,s,w)=0$, we get the result.
\end{example}
\begin{example} Let $X=[-1,1]\subseteq \mathbb{R}$ with
$G(x,y,z)=|x-y|+|y-z|+|x-z|$. Let $T:[-1,1]\rightarrow [-1,1]$ be
given by
$$
T(x)=\left\{ \begin{array}{ll}
-\frac{1}{2}xe^{-\frac{1}{|x|}},&x\in(0,1],\\
0,&x=0,\\
-\frac{1}{3}xe^{-\frac{1}{|x|}},&x\in[-1,0).
\end{array}\right.
$$
By taking $\psi(t,s,w)=0$, $\phi(t)=t$, and $x\in[0,1]$, $y\in
[-1,0]$, we have
\begin{eqnarray*}
G(Tx,Ty,Ty)&=&|Tx-Ty|+|Tx-Ty|+|Ty-Ty|\\
&=&|Tx-Ty|+|Tx-Ty|\\
&=&\left|-\frac{1}{2}xe^{-\frac{1}{|x|}}+\frac{1}{3}ye^{-\frac{1}{|y|}}\right|+\left|-\frac{1}{2}xe^{-\frac{1}{|x|}}+\frac{1}{3}ye^{-\frac{1}{|y|}}\right|\\
      &\le& \frac{1}{2}|x|+\frac{1}{3}|y|+\frac{1}{2}|x|+\frac{1}{3}|y|\\
      &\le&
\frac{1}{2}\left|x+\frac{1}{2}xe^{-\frac{1}{|x|}}\right|+\frac{1}{3}\left|y+\frac{1}{3}ye^{-\frac{1}{|y|}}\right|+\frac{1}{2}\left|x+\frac{1}{2}xe^{-\frac{1}{|x|}}\right|
+\frac{1}{3}\left|y+\frac{1}{3}ye^{-\frac{1}{|y|}}\right|\\
  &=
  &\frac{1}{2}|Tx-x|+\frac{1}{3}|Ty-y|+\frac{1}{2}|Tx-x|+\frac{1}{3}|Ty-y|\\
  &=
  &\frac{1}{2}(|Tx-x|+|Tx-x|)+\frac{1}{3}(|Ty-y|+|Ty-y|)\\
&=
  &\frac{1}{2}G(x,Tx,Tx)+\frac{1}{3}G(y,Ty,Ty),
\end{eqnarray*}
which implies that $T$ has a unique fixed point in
$[-1,0]\cap[0,1]$, namely $u=0$.
\end{example}

\end{document}